\documentclass[11pt,a4paper]{article}
\usepackage[utf8]{inputenc}
\usepackage{amsmath,amssymb,latexsym,color,enumerate,amsthm}
\usepackage{indentfirst}
\usepackage{authblk}
\usepackage[mathscr]{eucal}
\usepackage[colorlinks,linkcolor=blue,anchorcolor=green,citecolor=magenta]{hyperref}
\usepackage{setspace}

\newtheorem{thm}{Theorem}[section]

\newtheorem{lemma}[thm]{Lemma}

\newtheorem{conj}[thm]{Conjecture}

\newtheorem{example}{Example}[section]
\newtheorem{defin}{Definition}[section]

\newcommand{\KG}{\mathrm{KG}}

\newcommand{\calS}{\mathcal{S}}
\newcommand{\calT}{\mathcal{T}}
\newcommand{\calF}{\mathcal{F}}
\newcommand{\calA}{\mathcal{A}}
\newcommand{\calB}{\mathcal{B}}
\newcommand{\calC}{\mathcal{C}}
\newcommand{\calD}{\mathcal{D}}
\newcommand{\calG}{\mathcal{G}}
\newcommand{\calU}{\mathcal{U}}
\newcommand{\calV}{\mathcal{V}}
\newcommand{\calH}{\mathcal{H}}

\begin{document}

\renewcommand{\baselinestretch}{1.3}

\title{\bf Generalizations of the Erd\H{o}s Matching Conjecture for the $t$-Matching Number}

\author[1]{Mengyu Cao\thanks{E-mail: \texttt{myucao@ruc.edu.cn}}}
\author[2]{Mei Lu\thanks{E-mail: \texttt{lumei@tsinghua.edu.cn}}}
\author[2]{Haixiang Zhang\thanks{Corresponding author. E-mail: \texttt{zhang-hx22@mails.tsinghua.edu.cn}}}

\affil[1]{\small School of Mathematics, Renmin University of China, Beijing 100086, China}
\affil[2]{\small Department of Mathematical Sciences, Tsinghua University, Beijing 100084, China}

\date{}
\maketitle

\begin{abstract}
We write finite set systems as uniform hypergraphs.  A \emph{$t$-matching} in a $k$-uniform hypergraph is a set of hyperedges any two of which intersect in fewer than $t$ vertices.  The maximum size of such a set is the \emph{$t$-matching number} and is denoted by $\nu_t$.  We study the maximum number of hyperedges in a $k$-uniform hypergraph on $[n]$ with prescribed $t$-matching number.  This gives a hypergraph analogue of the Erd\H{o}s Matching Conjecture.  We also determine the second largest maximal structure with $\nu_t(\mathcal{F})=s$, extending work of Frankl and Kupavskii \cite{frankl2016two}. And, we obtain the extremal $G$-free induced subgraphs of generalized Kneser graph, generalizing Alishahi's results in \cite{alishahi2018extremal}.


\end{abstract}

\section{Introduction}\label{section 1}
Throughout the paper a \emph{$k$-graph} means a $k$-uniform hypergraph.  For a vertex set $V$ and an integer $r$, let $\binom{V}{r}$ denote the family of all $r$-subsets of $V$.  We write $[n]=\{1,2,\ldots,n\}$ and identify a $k$-graph on $[n]$ with its edge set $\calF\subseteq \binom{[n]}{k}$.

A $k$-graph $\calF$ is \emph{$t$-intersecting} if $|E\cap F|\ge t$ for all hyperedges $E,F\in\calF$.  It is \emph{trivial} if there is a fixed $t$-set $T\subseteq[n]$ contained in every edge of $\calF$, and \emph{non-trivial} otherwise.  The Erd\H{o}s--Ko--Rado theorem and its $t$-intersecting extensions determine the maximum number of hyperedges in a $t$-intersecting $k$-graph for large $n$: every extremal hypergraph consists of all $k$-edges containing a fixed $t$-set \cite{Erdos-Ko-Rado-1961-313,Frankl-1978,Wilson-1984}. This conjecture was partially resolved by Frankl and Füredi \cite{Frankl--Furedi-1991} and ultimately proved in full generality by Ahlswede and Khachatrian \cite{Ahlswede-Khachatrian-1997}.

The characterization of the maximum-sized non-trivial $t$-intersecting $k$-graphs was a long-standing problem. The foundational result in this direction is the Hilton-Milner Theorem \cite{Hilton-Milner-1967}, which completely describes such $k$-graph for the case $t=1$. Frankl made substantial progress in \cite{Frankl-1978-1} by solving the problem for $t \geq 2$ when $n > n_1(k,t)$. Using elegant shifting techniques, Frankl and Füredi \cite{Frankl-Furedi-1986} provided a beautiful new proof of the Hilton-Milner Theorem and raised the natural question of whether $n_1(k,t)$ grows linearly in $kt$. This question was answered affirmatively by Ahlswede and Khachatrian \cite{Ahlswede1996}, who in fact obtained a complete solution to all non-trivial intersection $k$-graphs.

\begin{thm}[Hilton--Milner \cite{Hilton-Milner-1967}; Ahlswede--Khachatrian \cite{Ahlswede1996}]\label{t-HM}
Let $\calF\subseteq\binom{[n]}{k}$ be a non-trivial $t$-intersecting $k$-graph.  If $k>t\ge2$ and $n>(t+1)(k-t+1)$, then
\begin{align*}
|\calF|&\le \max\{|\calH_1|,|\calH_2|\}\\
&=\max\left\{\binom{n-t}{k-t}-\binom{n-k-1}{k-t}+t,
(t+2)\binom{n-t-2}{k-t-1}+\binom{n-t-2}{k-t-2}\right\},
\end{align*}
where
\begin{align*}
\calH_1&=\left\{E\in\binom{[n]}{k}:[1,t]\subseteq E,\ E\cap[t+1,k+1]\ne\emptyset\right\}
\cup\left\{[1,k+1]\setminus\{i\}:1\le i\le k+1\right\},\\
\calH_2&=\left\{E\in\binom{[n]}{k}:|E\cap[1,t+2]|\ge t+1\right\}.
\end{align*}
For $t=1$, every non-trivial intersecting $k$-graph has at most the Hilton--Milner number of edges.
\end{thm}

A \emph{matching} in a $k$-graph is a set of pairwise disjoint hyperedges.  The classical matching number of $\calF$ is denoted by $\nu(\calF)$.  The Erd\H{o}s Matching Conjecture asks for the maximum number of hyperedges in a $k$-graph $\calF\subseteq\binom{[n]}{k}$ with $\nu(\calF)=s$.

\begin{conj}[Erd\H{o}s Matching Conjecture]
Let $\calF\subseteq\binom{[n]}{k}$ satisfy $\nu(\calF)=s$ and $n\ge ks+k-1$.  Then
\[
|\calF|\le \max\left\{\binom{n}{k}-\binom{n-s}{k},\binom{ks+k-1}{k}\right\}.
\]
Equivalently, the conjectured extremal examples are
\[
\calA_i=\left\{E\in\binom{[n]}{k}: |E\cap[is+i-1]|\ge i\right\},\qquad 1\le i\le k,
\]
with $\nu(\calA_i)=s$.
\end{conj}

For $n\ge(s+1)(k+1)$, the first construction is larger.  Large-$n$ bounds for this range were proved and improved by Bollob\'as, Daykin and Erd\H{o}s \cite{bollob1976sets}, Huang, Loh and Sudakov \cite{huang2012size}, Frankl, \L{}uczak and Mieczkowska \cite{frankl2012matchings}, Frankl \cite{frankl2013improved}, and Frankl and Kupavskii \cite{frankl2022}.

We now pass from ordinary matchings to $t$-matchings.  A \emph{$t$-matching} of size $m$ in a $k$-graph is a collection $\{E_1,\ldots,E_m\}$ of hyperedges such that $|E_i\cap E_j|<t$ whenever $i\ne j$.  The \emph{$t$-matching number} $\nu_t(\calF)$ is the maximum size of a $t$-matching in $\calF$.  Thus $\nu_1(\calF)=\nu(\calF)$, while $\nu_t(\calF)=1$ is exactly the condition that $\calF$ is $t$-intersecting.

A $k$-graph $\calF\subseteq\binom{[n]}{k}$ is \emph{maximal with respect to $\nu_t$} if every edge $E\in\binom{[n]}{k}\setminus\calF$ satisfies $\nu_t(\calF\cup\{E\})>\nu_t(\calF)$.  When $\nu_t(\calF)=1$, this is the usual notion of a maximal $t$-intersecting $k$-graph.

The standard shifting method is not monotone for $\nu_t$ when $t>1$.  For example, the $3$-graph
\[
\calF=\{123,124,125,126,134,234,345,346\}
\]
has $\nu_2(\calF)=2$, while a shifted version
\[
\calF'=\{123,124,125,126,134,135,136,234\}
\]
has $\nu_2(\calF')=3$.  Hence the proofs below use direct hypergraph arguments rather than shifting.

For $T\in\binom{[n]}{t}$, let
\[
\calS(T)=\left\{E\in\binom{[n]}{k}:T\subseteq E\right\}
\]
be the $t$-star with centre $T$.  If $\calT=\{T_1,\ldots,T_m\}\subseteq\binom{[n]}{t}$, put $\calS(\calT)=\bigcup_{i=1}^m\calS(T_i)$.  When $\calT$ consists of $m$ pairwise disjoint $t$-sets, define
\[
\ell(n,k,t,m)=|\calS(\calT)|,
\qquad \ell(n,k,t,0)=0.
\]
For fixed $m,k,t$, one has
\[
\ell(n,k,t,m)=m\binom{n-t}{k-t}+O(n^{k-t-1}).
\]

Pelekis and Rocha \cite{pelekis2018generalization} considered this problem.  In their notation, the step claiming $g(n,r,k,a)-\binom{n-k}{r-k}=g(n,r,k,a-1)$ is valid only in the range where the subtracted star has no overlap with the remaining extremal part.  Outside that range one only has the corresponding strict inequality, so the induction requires a different argument.

\begin{thm}\label{th1}
For every fixed $s\ge1$ and $1\le t<k$, there exists $N_0=N_0(k,t,s)$ such that the following holds for all $n\ge N_0$.  If $\calF\subseteq\binom{[n]}{k}$ is a $k$-graph with $\nu_t(\calF)=s$, then
\[
|\calF|\le \ell(n,k,t,s).
\]
The bound is attained by $\calS(\calT)$ whenever $\calT$ consists of $s$ pairwise disjoint $t$-sets. 
\end{thm}

A \emph{$t$-cover} of a $k$-graph $\calF$ is a set $\calC\subseteq\binom{[n]}{t}$ such that every edge of $\calF$ contains at least one member of $\calC$.  The \emph{$t$-covering number} $\tau_t(\calF)$ is the minimum size of a $t$-cover of $\calF$.  For $t=1$ this is the usual covering number.

Frankl and Kupavskii proved the following Hilton--Milner type theorem for ordinary matchings.

\begin{thm}[Frankl--Kupavskii \cite{frankl2016two}]\label{HM-type}
Suppose that $k\ge3$ and either $n\ge (s+\max\{25,2s+2\})k$ or $n\ge(2+o(1))sk$, where $o(1)$ is with respect to $s\to\infty$.  If $\calF\subseteq\binom{[n]}{k}$ satisfies $\nu(\calF)=s<\tau(\calF)$, then
\[
|\calF|\le |\calH^{(k)}(n,s)|,
\]
where
\begin{align*}
\calH^{(k)}(n,s)=&\left(\left\{E\in\binom{[n]}{k}:E\cap[s]\ne\emptyset\right\}\cup\{[s+1,s+k]\}\right)\\
&\setminus\left\{E\in\binom{[n]}{k}:E\cap[s]=\{s\},\ E\cap[s+1,s+k]=\emptyset\right\},
\end{align*}
and
\[
|\calH^{(k)}(n,s)|=\binom{n}{k}-\binom{n-k+s}{k}+1-\binom{n-s-k}{k-1}.
\]
\end{thm}

Our second result is the corresponding second-extremal statement for $t$-matchings.  Because the proof uses the large non-trivial $t$-intersecting structure theorem quoted in Lemma~\ref{t-stru}, we state it in the range $t\ge2$ and $k\ge t+2$.

\begin{thm}\label{t-HM-type}
Let $t\ge2$, $k\ge t+2$, and $s\ge2$ be fixed.  There is $N_1=N_1(k,t,s)$ such that for every $n\ge N_1$ and every $k$-graph $\calF\subseteq\binom{[n]}{k}$ with
\[
\nu_t(\calF)=s<\tau_t(\calF),
\]
one has
\[
|\calF|\le |\calH_t^{(k)}(n,s)|,
\]
where the extremal $k$-graph $\calH_t^{(k)}(n,s)$ is defined as follows.
\begin{enumerate}[\rm(i)]
\item If $t+2\le k\le 2t+1$, then
\begin{align*}
\calH_t^{(k)}(n,s)=&\left\{E\in\binom{[n]}{k}: [(r-1)t+1,rt]\subseteq E\text{ for some }1\le r\le s-1\right\}\\
&\cup\left\{E\in\binom{[n]}{k}: |E\cap[(s-1)t+1,st+2]|\ge t+1\right\}.
\end{align*}
\item If $k>2t+1$, then
\begin{align*}
\calH_t^{(k)}(n,s)=&\left\{E\in\binom{[n]}{k}: [(r-1)t+1,rt]\subseteq E\text{ for some }1\le r\le s-1\right\}\\
&\cup\left\{E\in\binom{[n]}{k}: [(s-1)t+1,st]\subseteq E,\ E\cap[st+1,(s-1)t+k+1]\ne\emptyset\right\}\\
&\cup\left\{[(s-1)t+1,(s-1)t+k+1]\setminus\{i\}:(s-1)t+1\le i\le (s-1)t+k+1\right\}.
\end{align*}
\end{enumerate}
\end{thm}

Theorem~\ref{t-HM-type} is a stability or second-extremal theorem in the following sense.  The primary extremal object in Theorem~\ref{th1} is a union of $s$ trivial $t$-stars.  The condition $\tau_t(\calF)>s$ excludes all such unions, because they have a $t$-cover of size at most $s$.  The hypergraph $\calH_t^{(k)}(n,s)$ is obtained by keeping $s-1$ trivial $t$-stars and replacing the last one by the largest non-trivial $t$-intersecting component.  Thus it is the next candidate after the primary extremal configuration is forbidden.

Let $h(n,k,t,s)=|\calH_t^{(k)}(n,s)|$.  For fixed $k,t,s$,
\[
h(n,k,t,s)=\ell(n,k,t,s-1)+\max\{|\calH_1|,|\calH_2|\}+O(n^{k-2t-1}),
\]
where $\calH_1,\calH_2$ are the two Hilton--Milner type $t$-intersecting hypergraphs in Theorem~\ref{t-HM}.  In particular $h(n,k,t,s)>\ell(n,k,t,s-1)$ for all sufficiently large $n$.

Let $G$ be a graph.  For $U\subseteq V(G)$, let $G[U]$ be the subgraph induced by $U$.  The Kneser graph $\KG_{n,k}$ has vertex set $\binom{[n]}{k}$, and two vertices are adjacent when the corresponding $k$-edges are disjoint.  Alishahi and Taherkhani \cite{alishahi2018extremal} studied the maximum size of a set $\calA\subseteq\binom{[n]}{k}$ for which the induced graph $\KG_{n,k}[\calA]$ contains no copy of a fixed graph $G$.

For a graph $G$ with chromatic number $\chi(G)=q$, define
\[
\eta(G)=\min\left\{\min_{1\le i\le q}|U_i|:(U_1,\ldots,U_q)\text{ is a proper }q\text{-coloring of }G\right\}.
\]
A subgraph $R$ of $G$ is called \emph{special} if deleting the vertices of $R$ lowers the chromatic number exactly by one, that is, $\chi(G-V(R))=\chi(G)-1$.

\begin{thm}[Alishahi--Taherkhani \cite{alishahi2018extremal}]\label{KG}
Let $k\ge2$ be fixed and let $G$ be a fixed graph with $\chi(G)=q$ and $\eta(G)=\eta$.  There exists $N(G,k)$ such that for every $n\ge N(G,k)$ and every $\calA\subseteq\binom{[n]}{k}$, if $\KG_{n,k}[\calA]$ contains no copy of $G$, then
\[
|\calA|\le \binom{n}{k}-\binom{n-q+1}{k}+\eta-1.
\]
Moreover, equality holds precisely for the usual construction consisting of a union $\calC$ of $q-1$ distinct $1$-stars together with $\eta-1$ additional vertices whose induced subgraph contains no special subgraph of $G$.
\end{thm}

Theorem \ref{KG} generalizes some intersection problems in $\binom{[n]}{k}$. For $G=K_2$, it asks for the largest intersecting family in $\binom{[n]}{k}$ \cite{Erdos-Ko-Rado-1961-313}. For $G=K_s$ ($s\geq3$), it's equivalent to Erd\H{o}s matching conjecture. For $G=K_{1,\ell}$, the case degenerates into $l$-almost intersecting family \cite{Gerbner2012}.
For $G=K_{s,t}$, the case degenerates into $(s,t)$-union intersecting family \cite{Katona2015}.

The \emph{generalized Kneser graph} $\KG_{n,k,t}$ has vertex set $\binom{[n]}{k}$, and two vertices are adjacent if the corresponding $k$-edges intersect in fewer than $t$ vertices.  Thus $\KG_{n,k,1}=\KG_{n,k}$.

\begin{thm}\label{th2}
Let $1\le t<k$ and let $G$ be a fixed graph with $\chi(G)=q$ and $\eta(G)=\eta$.  There exists $N_2=N_2(G,k,t)$ such that for every $n\ge N_2$ and every $\calA\subseteq\binom{[n]}{k}$, if $\KG_{n,k,t}[\calA]$ contains no copy of $G$, then
\[
|\calA|\le \ell(n,k,t,q-1)+\eta-1.
\]
Moreover, equality holds only in the following form: there is a collection $\calT\subseteq\binom{[n]}{t}$ with $|\calT|=q-1$ and $|\calS(\calT)|=\ell(n,k,t,q-1)$ such that
\[
\calS(\calT)\subseteq\calA,\qquad |\calA\setminus\calS(\calT)|=\eta-1,
\]
and $\KG_{n,k,t}[\calA\setminus\calS(\calT)]$ contains no copy of a special subgraph of $G$.  Conversely, every hypergraph of this form is $G$-free in $\KG_{n,k,t}$.  In particular, the construction applies whenever $\calT$ consists of $q-1$ pairwise disjoint $t$-sets.
\end{thm}

Since $\ell(n,k,1,q-1)=\binom{n}{k}-\binom{n-q+1}{k}$, Theorem~\ref{th2} recovers Theorem~\ref{KG}.  When $G=K_2$ it is the large-$n$ $t$-intersecting theorem; when $G=K_{s+1}$ it gives Theorem~\ref{th1}; and when $G=K_{1,s}$ it is an $s$-almost $t$-intersecting result.

The rest of the paper is organized as follows.  Section~\ref{sec2} proves Theorems~\ref{th1} and~\ref{t-HM-type}.  Section~\ref{sec3} proves Theorem~\ref{th2}.  Section~\ref{sec4} gives conjectures and a related question.

\section{Extremal structures of \texorpdfstring{$k$}{k}-graphs with prescribed \texorpdfstring{$t$}{t}-matching number}\label{sec2}
\subsection{The large-\texorpdfstring{$n$}{n} \texorpdfstring{$t$}{t}-matching theorem}
Throughout this section, all asymptotic notation is for $n\to\infty$ with $k,t,s$ fixed.

\begin{lemma}\label{star}
Let $n,k,t,m$ be positive integers with $1\le t<k$ and $n\ge mt$.  For every $\calT=\{T_1,\ldots,T_m\}\subseteq\binom{[n]}{t}$,
\[
|\calS(\calT)|\le \ell(n,k,t,m).
\]
\end{lemma}

\begin{proof}
It is enough to show that, whenever some centre $T_1$ meets the union of the other centres, replacing it by a $t$-set disjoint from all other centres does not decrease the size of the union of stars.  Repeating this operation gives a system of pairwise disjoint centres.

Choose $T_1^*\in\binom{[n]}{t}$ with $T_1^*\cap\bigcup_{i=2}^mT_i=\emptyset$, and put $\calT'=\{T_1^*,T_2,\ldots,T_m\}$.  Write
\[
T_1=R\cup S,
\qquad
T_1^*=R\cup T,
\qquad
R=T_1\cap T_1^*,
\]
where $S$ and $T$ are disjoint and have the same size.  Let $\phi:S\to T$ be a bijection.  Put
\[
\calA_1=\calS(\calT)\setminus\calS(\calT'),
\qquad
\calA_2=\calS(\calT')\setminus\calS(\calT).
\]
For $E\in\calA_1$, define
\[
\psi(E)=(E\setminus S)\cup T\cup\phi^{-1}(E\cap T).
\]
Since $E$ contains $S$ and does not contain all of $T$, the operation removes $|S|$ vertices and adds exactly $|T\setminus E|+|E\cap T|=|T|$ vertices.  Hence $\psi(E)$ is again a $k$-edge.  It contains $T_1^*$, and the only vertices added outside $E$ lie in $T_1^*$, which is disjoint from $T_i$ for $i\ge2$.  Thus $\psi(E)$ contains none of the unchanged centres $T_i$, $i\ge2$, because $E$ contained none of them.  Also $\psi(E)$ does not contain $T_1$, since $E$ did not contain all of $T$ and only the corresponding proper subset of $S$ is restored.  Therefore $\psi(E)\in\calA_2$.

If $\psi(E)=\psi(E')$, then $E\setminus S=E'\setminus S$.  Since both $E$ and $E'$ contain $S$, we get $E=E'$.  Thus $\psi$ is injective, and $|\calS(\calT)|\le |\calS(\calT')|$.  Iterating proves the lemma.
\end{proof}

\begin{lemma}\label{est1}
Let $A,B\in\binom{[n]}{k}$ satisfy $|A\cap B|<t$.  Define
\[
\calA(A)=\left\{E\in\binom{[n]}{k}:|E\cap A|\ge t\right\},
\qquad
\calA(B)=\left\{E\in\binom{[n]}{k}:|E\cap B|\ge t\right\}.
\]
Then
\[
|\calA(A)\cap\calA(B)|\le \binom{2k}{t+1}\binom{n-t-1}{k-t-1}.
\]
\end{lemma}

\begin{proof}
If $E\in\calA(A)\cap\calA(B)$, then
\[
|E\cap(A\cup B)|=|E\cap A|+|E\cap B|-|E\cap A\cap B|
\ge t+t-(t-1)=t+1.
\]
Choose a $(t+1)$-subset of $E\cap(A\cup B)$ and then choose the remaining $k-t-1$ vertices of $E$.  Since $|A\cup B|\le2k$, the claimed bound follows.
\end{proof}

\begin{lemma}\label{est2}
Let $A\in\binom{[n]}{k}$ and $T\in\binom{[n]}{t}$.  If $T\nsubseteq A$, then the number of edges in $\calS(T)$ that $t$-intersect $A$ is at most
\[
\binom{k+t}{t+1}\binom{n-t-1}{k-t-1}.
\]
\end{lemma}

\begin{proof}
If $E\in\calS(T)$ and $|E\cap A|\ge t$, then, because $T\nsubseteq A$,
\[
|E\cap(T\cup A)|\ge |E\cap T|+|E\cap A|-|T\cap A|
\ge t+t-(t-1)=t+1.
\]
As $|T\cup A|\le k+t$, the same counting argument as in Lemma~\ref{est1} gives the result.
\end{proof}

\begin{proof}[Proof of Theorem~\ref{th1}]
	We prove the theorem by induction on $s$.  The case $s=1$ is the large-$n$ Erd\H{o}s--Ko--Rado theorem for $t$-intersecting $k$-graphs.  Assume $s\ge2$ and that the result has been proved for $s-1$.
	
	Let $\{F_1,\ldots,F_s\}$ be a $t$-matching in $\calF$.  Since $\nu_t(\calF)=s$, every edge of $\calF$ $t$-intersects at least one $F_i$.  Define
	\begin{align*}
		\calA_i&=\{E\in\calF:|E\cap F_i|\ge t\},\\
		\calB_i&=\{E\in\calF:|E\cap F_i|\ge t\text{ and }|E\cap F_j|<t\text{ for all }j\ne i\}.
	\end{align*}
	
	First, $\calB_i$ is $t$-intersecting.  Indeed, if $E,E'\in\calB_i$ and $|E\cap E'|<t$, then
	\[
	\{E,E'\}\cup\{F_j:j\ne i\}
	\]
	is a $t$-matching of size $s+1$, a contradiction.
	
	Put $N=\binom{n-t}{k-t}$.  If $|\calA_i|<N/2$ for some $i$, then $\calF\setminus\calA_i$ has $t$-matching number $s-1$: it contains $\{F_j:j\ne i\}$, and any $t$-matching of size $s$ in $\calF\setminus\calA_i$ together with $F_i$ would give a $t$-matching of size $s+1$ in $\calF$.  By induction,
	\[
	|\calF|\le \ell(n,k,t,s-1)+N/2<\ell(n,k,t,s)
	\]
	for all sufficiently large $n$.  Hence we may assume $|\calA_i|\ge N/2$ for every $i$.
	
	Because the $F_i$ form a $t$-matching, Lemma~\ref{est1} applies to each pair $F_i,F_j$.  Therefore
	\begin{align}
		|\calB_i|
		&\ge |\calA_i|-\sum_{j\ne i}|\calA_i\cap\calA_j| \notag\\
		&\ge \frac{1}{2}N-(s-1)\binom{2k}{t+1}\binom{n-t-1}{k-t-1}.\label{Bi-large}
	\end{align}
	For sufficiently large $n$, the right-hand side is larger than the maximum size of a non-trivial $t$-intersecting $k$-graph: this follows from the Hilton--Milner theorem when $t=1$ and from Theorem~\ref{t-HM} when $t\ge2$.  Hence each $\calB_i$ is a trivial $t$-intersecting $k$-graph.
	
	Let $\calC_i=\calS(T_i)$ be the $t$-star containing $\calB_i$.  The centre is unique.  Indeed, if $\calB_i\subseteq\calS(T)\cap\calS(T')$ for two distinct $t$-sets $T,T'$, then $|T\cup T'|\ge t+1$, and hence
	\[
	|\calB_i|\le |\calS(T)\cap\calS(T')|\le \binom{n-t-1}{k-t-1}=O(n^{k-t-1}),
	\]
	contrary to the lower bound \eqref{Bi-large}.
	
	We claim that $\calF\subseteq\bigcup_{i=1}^s\calC_i$.  Suppose that $A_0\in\calF\setminus\bigcup_{i=1}^s\calC_i$.  We construct a $t$-matching $\{A_0,A_1,\ldots,A_s\}$ with $A_i\in\calB_i$.  At step $m+1$, where $0\le m<s$, we first check the hypothesis needed for Lemma~\ref{est2}.  By the choice of $A_0$, we have $A_0\notin\calC_{m+1}$, so $T_{m+1}\nsubseteq A_0$.  If $1\le r\le m$ and $T_{m+1}\subseteq A_r$, then, since $F_{m+1}\in\calB_{m+1}\subseteq\calS(T_{m+1})$, we would have
	\[
	|A_r\cap F_{m+1}|\ge t,
	\]
	contradicting the defining condition $A_r\in\calB_r$, namely $|A_r\cap F_j|<t$ for every $j\ne r$.  Thus $T_{m+1}\nsubseteq A_r$ for every previously chosen edge $A_r$.
	
	Therefore Lemma~\ref{est2} implies that the number of edges in $\calB_{m+1}$ that $t$-intersect one of $A_0,A_1,\ldots,A_m$ is at most
	\[
	(m+1)\binom{k+t}{t+1}\binom{n-t-1}{k-t-1}.
	\]
	For sufficiently large $n$, this quantity is smaller than the lower bound in \eqref{Bi-large}, so an admissible $A_{m+1}\in\calB_{m+1}$ exists.  This produces a $t$-matching of size $s+1$, a contradiction.  Thus $\calF\subseteq\bigcup_{i=1}^s\calC_i$.
	
	By Lemma~\ref{star},
	\[
	|\calF|\le\left|\bigcup_{i=1}^s\calC_i\right|\le\ell(n,k,t,s).
	\]
	If $\calT$ consists of $s$ pairwise disjoint $t$-sets, then $\nu_t(\calS(\calT))=s$ and $|\calS(\calT)|=\ell(n,k,t,s)$, so the bound is sharp.
\end{proof}

\subsection{The Hilton--Milner type stability theorem}

We next use a structure theorem for large maximal non-trivial $t$-intersecting $k$-graphs.

\noindent\textbf{Family I.}  Let $X\subseteq M\subseteq C\subseteq[n]$, where $|X|=t$, $|M|=k$, and $|C|=c\in\{k+1,k+2,\ldots,2k-t,n\}$.  Define
\[
\calH_1(X,M,C)=\calA(X,M)\cup\calB(X,M,C)\cup\calC(X,M,C),
\]
where
\begin{align*}
\calA(X,M)&=\{E\in\binom{[n]}{k}:X\subseteq E,\ |E\cap M|\ge t+1\},\\
\calB(X,M,C)&=\{E\in\binom{[n]}{k}:E\cap M=X,\ |E\cap C|=c-k+t\},\\
\calC(X,M,C)&=\{E\in\binom{C}{k}:|E\cap X|=t-1,\ |E\cap M|=k-1\}.
\end{align*}
When $c=k+1$, the hypergraph is isomorphic to $\calH_1$ in Theorem~\ref{t-HM}; in that case we write simply $\calH_1(X,C)$.

\noindent\textbf{Family II.}  If $Z\in\binom{[n]}{t+2}$, define
\[
\calH_2(Z)=\{E\in\binom{[n]}{k}:|E\cap Z|\ge t+1\}.
\]
This is isomorphic to $\calH_2$ in Theorem~\ref{t-HM}.

\begin{lemma}[Cao--Lv--Wang \cite{cao2021structure}]\label{t-stru}
Let $1\le t\le k-2$, and suppose
\[
n\ge (k-t+1)^2\max\left\{\binom{t+2}{2},\frac{k-t+2}{2}\right\}+t.
\]
If $\calF\subseteq\binom{[n]}{k}$ is a maximal non-trivial $t$-intersecting $k$-graph with
\[
|\calF|\ge (k-t)\binom{n-t-1}{k-t-1}-\binom{k-t}{2}\binom{n-t-2}{k-t-2},
\]
then either
\begin{enumerate}[\rm(i)]
\item $\calF=\calH_1(X,M,C)$ for some $X,M,C$ as above, or
\item $\calF=\calH_2(Z)$ for some $(t+2)$-set $Z$, and $k/2-1\le t\le k-2$.
\end{enumerate}
\end{lemma}

\begin{lemma}\label{t,s-constr}
Let $t\ge2$, $k\ge t+2$, and $s\ge2$ be fixed.  Let $\calF\subseteq\binom{[n]}{k}$ be maximal with respect to $\nu_t$, and suppose that
\[
|\calF|\ge h(n,k,t,s)>\ell(n,k,t,s-1),
\qquad
\tau_t(\calF)>\nu_t(\calF)=s.
\]
For all sufficiently large $n$, there are $t$-stars $\calC_2,\ldots,\calC_s$ and a maximal non-trivial $t$-intersecting $k$-graph $\calC_1$, isomorphic to $\calH_1$ or $\calH_2$, such that
\[
\calF\subseteq \calC_1\cup\calC_2\cup\cdots\cup\calC_s.
\]
\end{lemma}

\begin{proof}
	Let $\{F_1,\ldots,F_s\}$ be a $t$-matching in $\calF$, and define
	\begin{align*}
		\calA_i&=\{E\in\calF:|E\cap F_i|\ge t\},\\
		\calB_i&=\{E\in\calF:|E\cap F_i|\ge t\text{ and }|E\cap F_j|<t\text{ for every }j\ne i\}.
	\end{align*}
	As in the proof of Theorem~\ref{th1}, every $\calB_i$ is $t$-intersecting.
	
	Choose a fixed constant $c$ with $0<c<1/2$.  All asymptotic comparisons below are for fixed $k,t,s,c$ and $n\to\infty$.
	
	We shall repeatedly use the following consequence of Theorem~\ref{t-HM}.  If $|\calB_i|\ge c\binom{n-t}{k-t}$, then $\calB_i$ cannot be non-trivial, since every non-trivial $t$-intersecting $k$-graph has $O(n^{k-t-1})$ edges.  Thus $\calB_i$ is contained in a $t$-star.  Moreover, its centre is unique: if $\calB_i\subseteq\calS(T)\cap\calS(T')$ for distinct $t$-sets $T,T'$, then
	\[
	|\calB_i|\le\binom{n-t-1}{k-t-1}=O(n^{k-t-1}),
	\]
	contradicting $|\calB_i|\ge c\binom{n-t}{k-t}$.
	
	\noindent\emph{Case 1: all $\calB_i$ are large.}  Suppose $|\calB_i|\ge c\binom{n-t}{k-t}$ for all $i$.  By the preceding paragraph, every $\calB_i$ is contained in a unique $t$-star $\calC_i$.  The proof of Claim 2 in Theorem~\ref{th1}, with $c\binom{n-t}{k-t}$ in place of \eqref{Bi-large} and with the same verification of the hypothesis of Lemma~\ref{est2}, gives $\calF\subseteq\bigcup_{i=1}^s\calC_i$.  Hence $\tau_t(\calF)\le s$, contrary to the hypothesis.
	
	\noindent\emph{Case 2: at least two $\calB_i$ are small.}  Assume $|\calB_1|,|\calB_2|<c\binom{n-t}{k-t}$.  Lemma~\ref{est1} gives, for $i=1,2$,
	\[
	|\calA_i|<c\binom{n-t}{k-t}+(s-1)\binom{2k}{t+1}\binom{n-t-1}{k-t-1}.
	\]
	Therefore
	\begin{align*}
		|\calF\setminus(\calA_1\cup\calA_2)|
		&\ge h(n,k,t,s)-2c\binom{n-t}{k-t}-O(n^{k-t-1})\\
		&=(s-1-2c)\binom{n-t}{k-t}+O(n^{k-t-1})\\
		&>\ell(n,k,t,s-2)
	\end{align*}
	for sufficiently large $n$, since $c<1/2$.  By Theorem~\ref{th1}, $\calF\setminus(\calA_1\cup\calA_2)$ contains a $t$-matching of size $s-1$.  Together with $F_1$ and $F_2$ this gives a $t$-matching of size $s+1$, impossible.
	
	\noindent\emph{Case 3: exactly one $\calB_i$ is small.}  Assume $|\calB_1|<c\binom{n-t}{k-t}$ and $|\calB_i|\ge c\binom{n-t}{k-t}$ for $2\le i\le s$.  For each $i\ge2$, let $\calC_i=\calS(T_i)$ be the unique $t$-star containing $\calB_i$.
	
	We first show that
	\[
	\calD:=\calF\setminus\bigcup_{i=2}^s\calC_i
	\]
	is $t$-intersecting.  If not, take $G_1,G_2\in\calD$ with $|G_1\cap G_2|<t$.  Inductively choose $A_i\in\calB_i$ for $2\le i\le s$ so that $\{G_1,G_2,A_2,\ldots,A_i\}$ is a $t$-matching.
	
	At step $i$, the hypothesis of Lemma~\ref{est2} is satisfied for every previously chosen edge.  For $G_1$ and $G_2$ this follows from $G_1,G_2\in\calD$, since $\calD$ is disjoint from $\calC_i=\calS(T_i)$.  If $2\le r<i$ and $T_i\subseteq A_r$, then $F_i\in\calB_i\subseteq\calS(T_i)$ gives $|A_r\cap F_i|\ge t$, contradicting $A_r\in\calB_r$.  Hence Lemma~\ref{est2} forbids at most
	\[
	i\binom{k+t}{t+1}\binom{n-t-1}{k-t-1}<c\binom{n-t}{k-t}\le |\calB_i|
	\]
	edges of $\calB_i$ for sufficiently large $n$.  This yields a $t$-matching of size $s+1$, a contradiction.  Hence $\calD$ is $t$-intersecting.
	
	Let $\calC_1$ be a maximal $t$-intersecting $k$-graph containing $\calD$.  It is non-trivial; otherwise a single $t$-set covering $\calC_1$, together with the centres $T_2,\ldots,T_s$, would be a $t$-cover of $\calF$ of size at most $s$.  Also
	\begin{align*}
		|\calC_1|&\ge |\calD|\ge |\calF|-\left|\bigcup_{i=2}^s\calC_i\right|\\
		&\ge h(n,k,t,s)-\ell(n,k,t,s-1)\\
		&=\max\{|\calH_1|,|\calH_2|\}+O(n^{k-2t-1})\\
		&\ge (k-t+1)\binom{n-t-1}{k-t-1}+O(n^{k-t-2}).
	\end{align*}
	This is larger than the threshold in Lemma~\ref{t-stru} for sufficiently large $n$.  Lemma~\ref{t-stru} applies.  In the $\calH_1(X,M,C)$ case, if $|C|=k+1$, then $\calC_1$ is the standard $\calH_1$-type hypergraph.  If $|C|>k+1$ and either $C\ne[n]$ or $k\ge t+3$, then a direct count gives
	\[
	|\calH_1(X,M,C)|\le (k-t)\binom{n-t-1}{k-t-1}+O(n^{k-t-2}),
	\]
	contradicting the preceding lower bound.  The remaining exceptional case is $C=[n]$ and $k=t+2$; then
	\[
	\calH_1(X,M,[n])=\{E\in\binom{[n]}{k}: |E\cap M|\ge k-1\}=\calH_2(M),
	\]
	since $|M|=k=t+2$.  Thus $\calC_1$ is isomorphic to $\calH_1$ or $\calH_2$.
	
	Since every edge of $\calF$ either lies in one of the stars $\calC_2,\ldots,\calC_s$ or lies in $\calD\subseteq\calC_1$, we have $\calF\subseteq\bigcup_{i=1}^s\calC_i$.
\end{proof}

\begin{proof}[Proof of Theorem~\ref{t-HM-type}]
	Let $\calF$ be a largest $k$-graph satisfying $\nu_t(\calF)=s<\tau_t(\calF)$.  We may assume $\calF$ is maximal with respect to $\nu_t$: if an edge can be added without increasing $\nu_t$ beyond $s$, add it; the condition $\tau_t>s$ is preserved, since any $t$-cover of the larger hypergraph would also cover the original one.
	
	The construction $\calH_t^{(k)}(n,s)$ satisfies $\nu_t=s<\tau_t$.  Indeed, it is the union of $s-1$ full $t$-stars and one non-trivial $t$-intersecting Hilton--Milner component.  A $t$-matching uses at most one edge from each component, so $\nu_t\le s$, and equality is obtained by choosing one edge from each component with all auxiliary vertices disjoint, which is possible for sufficiently large $n$.  Also, any $t$-cover of size at most $s$ must contain the centre of each of the $s-1$ full stars; otherwise one can choose an edge of that star avoiding all proposed cover $t$-sets.  The remaining Hilton--Milner component is non-trivial and hence is not covered by a single $t$-set.  Thus $\tau_t>s$.
	
	Consequently $|\calF|\ge h(n,k,t,s)>\ell(n,k,t,s-1)$.  Lemma~\ref{t,s-constr} gives
	\[
	\calF\subseteq\calG_1(T_1,\ldots,T_{s-1};X,C)
	\quad\text{or}\quad
	\calF\subseteq\calG_2(T_1,\ldots,T_{s-1};Z).
	\]
	By Lemma~\ref{shift}, the maximum possible size of either union is attained when the $T_i$ are pairwise disjoint and disjoint from the relevant core $C$ or $Z$.
	
	In this disjoint position, the overlap between the $s-1$ full stars and the Hilton--Milner component is negligible at the precision needed below.  For example, if a star centre $T_i$ is disjoint from the core $Z$ of $\calH_2(Z)$, then every edge of $\calS(T_i)\cap\calH_2(Z)$ contains $T_i$ and at least $t+1$ vertices of the fixed set $Z$; hence this intersection is empty when $k<2t+1$ and has size $O(n^{k-2t-1})$ when $k\ge2t+1$.  Similarly, an edge of $\calS(T_i)\cap\calH_1(X,C)$ must contain $T_i$, $X$, and at least one additional vertex from the fixed core $C\setminus X$, apart from finitely many exceptional edges.  Intersections involving two or more full stars are smaller.  Therefore the total overlap is $o(n^{k-t-2})$, since $t\ge2$.
	
	Thus
	\begin{align*}
		|\calG_1|&=\ell(n,k,t,s-1)+|\calH_1|+o(n^{k-t-2}),\\
		|\calG_2|&=\ell(n,k,t,s-1)+|\calH_2|+o(n^{k-t-2}).
	\end{align*}
	
	Let $m=k-t$.  Using
	\[
	\binom{n-a}{b}=\frac{n^b}{b!}-\frac{2a+b-1}{2(b-1)!}n^{b-1}+O(n^{b-2})
	\]
	for fixed $a,b$, we obtain
	\[
	|\calH_1|-|\calH_2|=\frac{k-2t-1}{(k-t-1)!}n^{k-t-1}+O(n^{k-t-2}).
	\]
	Thus $\calH_1$ is larger when $k>2t+1$, while $\calH_2$ is larger when $k<2t+1$.  In the boundary case $k=2t+1$, the leading term cancels and the next term is
	\[
	|\calH_1|-|\calH_2|=-\frac{t(t+1)}{2(t-1)!}n^{t-1}+O(n^{t-2})<0
	\]
	for all sufficiently large $n$.  Therefore the largest possible union is exactly $\calH_t^{(k)}(n,s)$ as defined in the theorem, and $|\calF|\le |\calH_t^{(k)}(n,s)|$.
\end{proof}

\section{Extremal \texorpdfstring{$G$}{G}-free induced subgraphs of generalized Kneser graphs}\label{sec3}

Throughout this section, all asymptotic notation is for $n\to\infty$ with $G,k,t$ fixed.

\begin{proof}[Proof of Theorem~\ref{th2}]
	If $q=1$, then $G$ is an edgeless graph on $\eta=|V(G)|$ vertices.  Since subgraphs are not required to be induced, every set $\calA$ with $|\calA|\ge\eta$ contains a copy of $G$.  Hence a $G$-free $\calA$ has $|\calA|\le\eta-1=\ell(n,k,t,0)+\eta-1$, and the equality statement is immediate with the empty collection of centres.  We may therefore assume $q\ge2$.
	
	Fix a proper $q$-coloring $V(G)=V_0\cup\cdots\cup V_{q-1}$ with $|V_0|=\eta$ and write $a_i=|V_i|$ and $a=|V(G)|$.  It is enough to construct a complete $q$-partite subgraph with parts of sizes $a_0,\ldots,a_{q-1}$, because such a graph contains a copy of $G$.
	
	Choose $\varepsilon>0$ sufficiently small, depending only on $G,k,t$, and put $d=\varepsilon n^{k-t}$.  Then choose $n$ sufficiently large so that
	\begin{enumerate}[(a)]
		\item $d$ dominates every constant multiple of $\binom{n-t-1}{k-t-1}$ appearing below;
		\item $a\binom{k}{t}d<\ell(n,k,t,q-1)-\ell(n,k,t,q-2)$.
	\end{enumerate}
	For $\calA\subseteq\binom{[n]}{k}$, define the set of dense $t$-centres
	\[
	\calT=\left\{T\in\binom{[n]}{t}:|\calS(T)\cap\calA|\ge d\right\}.
	\]
	Assume throughout that $\KG_{n,k,t}[\calA]$ is $G$-free.
	
	\noindent\emph{Case 1: $|\calT|\le q-2$.}  We build a clique of size $a$ inside $\KG_{n,k,t}[\calA\setminus\calS(\calT)]$.  Suppose $\calU_{i-1}$ has been chosen with $|\calU_{i-1}|=i-1$.  Define
	\[
	\calV_i=\{E\in\calA\setminus\calS(\calT): |E\cap U|\ge t\text{ for some }U\in\calU_{i-1}\}.
	\]
	For every $U\in\calU_{i-1}$ and every $W\in\binom{U}{t}$, one has $W\notin\calT$, so $|\calS(W)\cap\calA|<d$.  Therefore
	\[
	|\calV_i|\le \sum_{U\in\calU_{i-1}}\sum_{W\in\binom{U}{t}}|\calS(W)\cap\calA|
	\le (i-1)\binom{k}{t}d.
	\]
	If $|\calA|\ge\ell(n,k,t,q-1)+\eta-1$, then by Lemma~\ref{star}
	\[
	|\calA\setminus(\calS(\calT)\cup\calV_i)|
	\ge \ell(n,k,t,q-1)-\ell(n,k,t,q-2)-a\binom{k}{t}d>0.
	\]
	Thus we can choose the next vertex.  After $a$ steps we obtain a clique $K_a$, which contains $G$, a contradiction.
	
	\noindent\emph{Case 2: $|\calT|\ge q$.}  Choose distinct $T_0,\ldots,T_{q-1}\in\calT$ and put
	\[
	\calS^*(T_i)=\left(\calS(T_i)\cap\calA\right)\setminus\bigcup_{j\ne i}\calS(T_j).
	\]
	Since two distinct $t$-sets have union of size at least $t+1$,
	\[
	|\calS^*(T_i)|\ge d-(q-1)\binom{n-t-1}{k-t-1}.
	\]
	We choose the $i$th color class inside $\calS^*(T_i)$.  Suppose classes $0,\ldots,i-1$ have already been chosen and let $\calU_{i-1}$ be their union.  Set
	\[
	\calV_i=\{E\in\calS^*(T_i): |E\cap U|\ge t\text{ for some }U\in\calU_{i-1}\}.
	\]
	For each previously chosen $U$, the centre $T_i$ is not contained in $U$, so Lemma~\ref{est2} bounds the number of forbidden edges by
	\[
	|\calV_i|\le a\binom{k+t}{t+1}\binom{n-t-1}{k-t-1}.
	\]
	Hence $|\calS^*(T_i)\setminus\calV_i|\ge a_i$ for sufficiently large $n$.  This constructs a complete $q$-partite graph with parts $a_0,\ldots,a_{q-1}$, again a contradiction.
	
	\noindent\emph{Case 3: $|\calT|=q-1$.}  Write $\calT=\{T_1,\ldots,T_{q-1}\}$.  Lemma~\ref{star} gives $|\calS(\calT)|\le\ell(n,k,t,q-1)$.  Hence
	\[
	|\calA\setminus\calS(\calT)|\ge |\calA|-|\calS(\calT)|\ge \eta-1.
	\]
	If $|\calA\setminus\calS(\calT)|\ge\eta$, choose $\eta$ vertices from $\calA\setminus\calS(\calT)$ to form the first color class.  For $1\le i\le q-1$, define $\calS^*(T_i)$ as in Case 2, with the other $q-2$ centres removed.  Then
	\[
	|\calS^*(T_i)|\ge d-(q-2)\binom{n-t-1}{k-t-1}.
	\]
	When the previous color classes have union $\calU_{i-1}$, every $U\in\calU_{i-1}$ satisfies $T_i\nsubseteq U$: this is clear for the initial class outside $\calS(\calT)$, and it follows from the definition of $\calS^*(T_j)$ for classes chosen from other centres.  Hence Lemma~\ref{est2} gives
	\[
	\left|\{E\in\calS^*(T_i): |E\cap U|\ge t\text{ for some }U\in\calU_{i-1}\}\right|
	\le a\binom{k+t}{t+1}\binom{n-t-1}{k-t-1}.
	\]
	Thus sufficiently many choices remain for the $i$th color class.  This gives a copy of $G$, a contradiction.
	
	Therefore every $G$-free $\calA$ satisfies
	\[
	|\calA|\le\ell(n,k,t,q-1)+\eta-1.
	\]
	This proves the upper bound.
	
	It remains to identify the equality case.  Suppose $|\calA|=\ell(n,k,t,q-1)+\eta-1$ and $\KG_{n,k,t}[\calA]$ is $G$-free.  The previous cases force $|\calT|=q-1$ and $|\calA\setminus\calS(\calT)|=\eta-1$.  Consequently
	\[
	|\calA\cap\calS(\calT)|=\ell(n,k,t,q-1).
	\]
	Since $|\calA\cap\calS(\calT)|\le |\calS(\calT)|\le\ell(n,k,t,q-1)$, we get $\calS(\calT)\subseteq\calA$ and $|\calS(\calT)|=\ell(n,k,t,q-1)$.
	
	If $\KG_{n,k,t}[\calA\setminus\calS(\calT)]$ contained a copy $R'$ of a special subgraph $R_0\subseteq G$, then $\chi(G-V(R_0))=q-1$.  Let the vertices of this copy $R'$ be $R_1,\ldots,R_r$, viewed as $k$-edges in $\calA\setminus\calS(\calT)$.  Since $\KG_{n,k,t}[\calS(\calT)]$ is $(q-1)$-colorable by assigning each edge to one centre it contains, it suffices to embed $G-V(R_0)$ into a large complete $(q-1)$-partite subgraph of $\KG_{n,k,t}[\calS(\calT)]$ whose vertices are adjacent to all $R_j$.
	
	Choose pairwise disjoint sets $X_1,\ldots,X_{q-1}$ in
	\[
	[n]\setminus\left(\bigcup_{j=1}^r R_j\cup\bigcup_{i=1}^{q-1}T_i\right)
	\]
	large enough that each $X_i$ contains $a$ pairwise disjoint $(k-t)$-sets $Y_i^1,\ldots,Y_i^a$.  The $k$-edges $T_i\cup Y_i^j$ lie in $\calS(\calT)\subseteq\calA$.  They form a complete $(q-1)$-partite graph with parts indexed by $i$, because distinct centres intersect in fewer than $t$ vertices and all auxiliary sets are disjoint.  Moreover, each such edge is adjacent to every $R_j$, since $R_j\notin\calS(\calT)$ implies $T_i\nsubseteq R_j$ and the auxiliary sets avoid $R_j$.  This embeds $G$, a contradiction.  Hence $\calA\setminus\calS(\calT)$ contains no special subgraph of $G$.
	
	Conversely, assume $\calS(\calT)\subseteq\calA$, $|\calS(\calT)|=\ell(n,k,t,q-1)$, $|\calA\setminus\calS(\calT)|=\eta-1$, and $\KG_{n,k,t}[\calA\setminus\calS(\calT)]$ contains no special subgraph of $G$.  If a copy $Q$ of $G$ existed in $\KG_{n,k,t}[\calA]$, then $Q[V(Q)\cap\calS(\calT)]$ would be $(q-1)$-colorable.  List the vertices of $V(Q)\setminus\calS(\calT)$ as $x_1,\ldots,x_m$, and put $Q_j=Q-\{x_1,\ldots,x_j\}$.  Then $\chi(Q_0)=q$ and $\chi(Q_m)\le q-1$.  Since deleting one vertex decreases chromatic number by at most one, there is a smallest $j$ such that $\chi(Q_j)\le q-1$, and necessarily $\chi(Q_j)=q-1$.  The subgraph of $Q$ induced by $\{x_1,\ldots,x_j\}$ is therefore a special subgraph of the copy $Q\cong G$, and it lies in $\calA\setminus\calS(\calT)$, a contradiction.  Thus $\calA$ is $G$-free.
\end{proof}

\section{Remarks and open problems}\label{sec4}
The direct proofs above avoid shifting, because shifting need not decrease $\nu_t$ for $t>1$.  It remains natural to ask whether the large-$n$ thresholds in Theorems~\ref{th1}, \ref{t-HM-type}, and \ref{th2} can be made linear in $k$.

\begin{conj}
Let $1\le t<k$ and $s\ge2$.  There exists a constant $c=c(t,s)$ such that, whenever $n\ge ck$ and $\calF\subseteq\binom{[n]}{k}$ satisfies $\nu_t(\calF)=s$, one has
\[
|\calF|\le\ell(n,k,t,s).
\]
Moreover, the bound is attained by $\calS(\calT)$ for any $s$ pairwise disjoint $t$-sets $\calT$.
\end{conj}

\begin{conj}
Suppose $k>t\ge2$ and $s\ge2$.  There exists a constant $c=c(t,s)$ such that, whenever $n\ge ck$ and $\calF\subseteq\binom{[n]}{k}$ satisfies $\tau_t(\calF)>\nu_t(\calF)=s$, one has
\[
|\calF|\le |\calH_t^{(k)}(n,s)|.
\]
\end{conj}

Gerbner et al.\ \cite{Gerbner2019} established a stability version of Theorem~\ref{KG}.  Given a $k$-graph $\calF$ and an integer $r\ge2$, let $\ell_r(\calF)$ be the minimum number of hyperedges whose deletion makes the remaining $k$-graph contain no ordinary matching of size $r$.

\begin{thm}[Gerbner et al.\ \cite{Gerbner2019}]
For any $k\ge2$ and integers $s_1>s_2>\cdots>s_r>s_{r+1}\ge1$, there exists $n_0=n_0(k,s_1,\ldots,s_{r+1})$ such that if $n\ge n_0$, $\calF\subseteq\binom{[n]}{k}$, $\ell_{r+1}(\calF)\ge s_{r+1}$, and $\KG_{n,k}[\calF]$ is $K_{s_1,\ldots,s_{r+1}}$-free, then
\[
|\calF|\le\binom{n}{k}-\binom{n-s_1}{k}+\binom{s_1}{2}-\binom{s_{r+1}}{2}+r+s_{r+1}-1.
\]
\end{thm}

For $s,t\ge1$, let $\ell_{s,t}(\calF)$ be the minimum number of hyperedges whose deletion makes the remaining $k$-graph contain no $t$-matching of size $s$.

\noindent\textbf{Question.}  Determine the maximum size and the extremal structures of $k$-graphs $\calF\subseteq\binom{[n]}{k}$ such that $\ell_{r+1,t}(\calF)\ge s_{r+1}$ and $\KG_{n,k,t}[\calF]$ is $K_{s_1,\ldots,s_{r+1}}$-free, for fixed integers $s_1>s_2>\cdots>s_r>s_{r+1}\ge1$ and sufficiently large $n$.

\section*{Acknowledgement}
M. Cao is supported by the National Natural Science Foundation of China (12301431) and Beijing Natural Science Foundation (1262010)..  M. Lu is supported by the National Natural Science Foundation of China (Grant 12571372) and Beijing Natural Science Foundation (1262010)..

\end{document}